\documentclass[reqno]{amsart}

\usepackage{amsmath,amssymb,amsthm,amsfonts,mathrsfs}
\usepackage{fullpage}
\usepackage{xcolor}
\usepackage[colorlinks=true,
linkcolor=blue,
anchorcolor=blue,
citecolor=red
]{hyperref}
\allowdisplaybreaks 
\newtheorem{theorem}{Theorem}[section]
\newtheorem{lemma}[theorem]{Lemma}

\newtheorem{conjecture}[theorem]{Conjecture}
\newtheorem*{conjecture*}{Conjecture}
\theoremstyle{definition}

\theoremstyle{remark}

\newtheorem*{remark*}{remark}

\usepackage{colonequals}

\author{Runbo Li}
\address{The High School Affiliated to Renmin University of China, Beijing 100080, People's Republic of China}
\email{carey.lee.0433@gmail.com}

\makeatletter
\@namedef{subjclassname@2020}{\textup{}2020 Mathematics Subject Classification}
\makeatother

\title[]{On the upper and lower bound orders of almost prime triples}

\begin{document}
	
\begin{abstract}
A Hardy-Littlewood triple is a 3-tuple of integers with the form $(n, n+2, n+6)$. In this paper, we study Hardy-Littlewood triples of the form $(p, P_{a}, P_{b})$ and improve the upper and lower bound orders of it, where $p$ is a prime and $P_{r}$ has at most $r$ prime factors. Our new results generalize and improve the previous results.
\end{abstract}

\maketitle


\tableofcontents

\section{Introduction}
Let $x$ be a sufficiently large integer, $N$ be a sufficiently large even integer, $p$ be a prime, and let $P_{r}$ denote an integer with at most $r$ prime factors counted with multiplicity. For each $N \geqslant 4$ and $r \geqslant 2$, we define 
\begin{equation}
\pi_{1,r}(x):= \left|\left\{p : p \leqslant x, p+2=P_{r}\right\}\right|
\end{equation}
and\begin{equation}
D_{1,r}(N):= \left|\left\{p : p \leqslant N, N-p=P_{r}\right\}\right|.
\end{equation}

In 1966 Jingrun Chen \cite{Chen1966} proved his remarkable Chen's theorem: let $x$ be a sufficiently large integer and $N$ be a sufficiently large even integer, then
\begin{equation}
\pi_{1,2}(x)\gg \frac{C_2 x}{(\log x)^2} \quad \text{and} \quad D_{1,2}(N) \gg \frac{C(N) N}{(\log N)^2},
\end{equation}
where
\begin{equation}
C_2:=2\prod_{p>2}\left(1-\frac{1}{(p-1)^{2}}\right) \quad \text{and} \quad C(N):=\prod_{\substack{p \mid N \\ p>2}} \frac{p-1}{p-2} \prod_{p>2}\left(1-\frac{1}{(p-1)^{2}}\right)
\end{equation}
and the detail was published in \cite{Chen1973}. In 1990, Wu \cite{WU90} generalized Chen's theorem and proved that
\begin{equation}
D_{1,3}(N) \gg \frac{C(N) N}{(\log N)^2} \log \log N 
\end{equation}
and
\begin{equation}
D_{1,r}(N) \gg \frac{C(N) N}{(\log N)^2} (\log \log N)^{r-2}. 
\end{equation}
and Kan \cite{Kan1} proved the similar result in 1991. Kan \cite{Kan2} also proved the more generalized theorem in 1992:
\begin{equation}
D_{s,r}(N) \gg \frac{C(N) N}{(\log N)^2} (\log \log N)^{s+r-3},
\end{equation}
where $s \geqslant 1$,
\begin{equation}
D_{s,r}(N):= \left|\left\{P_s : P_s \leqslant N, N-P_s=P_{r}\right\}\right|.
\end{equation}
Clearly their methods can be modified to get a similar lower bound order on the twin prime version. For this, we refer the interested readers to \cite{Kan3}.

Now we focus on the Hardy-Littlewood triples $(n, n+2, n+6)$ with almost-prime values. In fact, if we define
\begin{equation}
\pi_{1,a,b}(x):= \left|\left\{p : p \leqslant x, p+2=P_{a}, p+6=P_{b}\right\}\right|
\end{equation}
and
\begin{equation}
D_{1,a,b}(N):= \left|\left\{p : p \leqslant N, N-p=P_{a}, p+6=P_{b}\right\}\right|,
\end{equation}
then a special case of Hardy-Littlewood conjecture states that $\pi_{1,1,1}(x)$ should be asymptotic to $\frac{C_3 x}{(\log x)^3}$, where
\begin{equation}
C_3=\frac{9}{2} \prod_{p>3}\left(1-\frac{3 p-1}{(p-1)^{3}}\right) \approx 2.86259.
\end{equation}
In 2015, Heath-Brown and Li \cite{HBL2016} proved that $\pi_{1,2,76}(x) \gg \frac{C_3 x}{(\log x)^3}$ and Cai \cite{Cai2017} improved this result to $\pi_{1,2,14}(x) \gg \frac{C_3 x}{(\log x)^3}$ by using a delicate sieve process later. Their results refined Chen's theorem. Like Wu's generalization of Chen's theorem, we may conjecture that $\pi_{1,3,r}(x)$ should be asymptotic to $\frac{C_3 x\log \log x}{(\log x)^3}$ for some large $r$.
Very recently, Li and Liu \cite{ljm} proved $\pi_{1,3,6}(x) \gg \frac{C_3 x}{(\log x)^3}$. They also got $\pi_{1,3,3}(x) \gg \frac{C_3 x}{(\log x)^3}$ by assuming GEH(0.99). In this paper, we improve their asymptotic estimates of $\pi_{1,3,r}(x)$ on the orders by fixing some small prime factors $q$ and prove that:
\begin{theorem}\label{t1}
For every integer $a \geqslant 2$ and $b \geqslant 14$, we have
$$
\pi_{1,a,b}(x) \gg \frac{C_3 x}{(\log x)^3}(\log \log x)^{a-2}
\quad and \quad 
D_{1,a,b}(N) \gg \frac{N}{(\log N)^3}(\log \log N)^{a-2},
$$
where $\pi_{1,a,b}(x)$ and $D_{1,a,b}(N)$ are defined above.
\end{theorem}
By similar arguments, we also obtain those theorems:
\begin{theorem}\label{t2}
For every integer $a \geqslant 3$ and $b \geqslant 6$, we have
$$
\pi_{1,a,b}(x) \gg \frac{C_3 x}{(\log x)^3}(\log \log x)^{a-3}
\quad and \quad 
D_{1,a,b}(N) \gg \frac{N}{(\log N)^3}(\log \log N)^{a-3}.
$$
\end{theorem}
\begin{theorem}\label{t3}
For every integer $a \geqslant 4$ and $b \geqslant 5$, we have
$$
\pi_{1,a,b}(x) \gg \frac{C_3 x}{(\log x)^3}(\log \log x)^{a-4}
\quad and \quad 
D_{1,a,b}(N) \gg \frac{N}{(\log N)^3}(\log \log N)^{a-4}.
$$
\end{theorem}
We also prove some conditional results:
\begin{theorem}\label{t4}
For every integer $a \geqslant 2$ and $b \geqslant 4$, assuming GEH(0.99), we have
$$
\pi_{1,a,b}(x) \gg \frac{C_3 x}{(\log x)^3}(\log \log x)^{a-2}
\quad and \quad 
D_{1,a,b}(N) \gg \frac{N}{(\log N)^3}(\log \log N)^{a-2}.
$$
\end{theorem}
\begin{theorem}\label{t5}
For every integer $a \geqslant 3$ and $b \geqslant 3$, assuming GEH(0.99), we have
$$
\pi_{1,a,b}(x) \gg \frac{C_3 x}{(\log x)^3}(\log \log x)^{a-3}
\quad and \quad 
D_{1,a,b}(N) \gg \frac{N}{(\log N)^3}(\log \log N)^{a-3}.
$$
\end{theorem}
While calculating, we find that it seems hard to improve our Theorems~\ref{t1}--\ref{t5}. (For example, you need to get an improvement about 45 percent on the sieve process to replace the condition $b \geqslant 14$ by $b \geqslant 13$ in our Theorem 1.)

In this paper, we only provide a detailed proof of $\pi_{1,3,14}(x) \gg \frac{C_3 x \log \log x}{(\log x)^3}$, which is a simple version of Theorem~\ref{t1}. The readers can modify our proof to get Theorems~\ref{t1}--\ref{t5}.

\section{Preliminary lemmas}
Let $\mathcal{A}$ denote a finite set of positive integers, $\mathcal{P}$ denote an infinite set of primes, $q$ denote a prime number satisfies $q< x^\varepsilon$ and put
$$
\mathcal{A}=\left\{\frac{p+2}{q}: 7<p \leqslant x, p \equiv -2 (\bmod q), (p+6, P(z))=1\right\}, \quad \mathcal{P}=\{p : (p,q)=1\},
$$
$$
\mathcal{P}(r)=\{p : p \in \mathcal{P},(p, r)=1\}, \quad 
P(z)=\prod_{\substack{p\in \mathcal{P}\\p<z}} p,
$$
$$
\mathcal{A}_{d}=\{a : a \in \mathcal{A}, a \equiv 0(\bmod d)\},\quad
S(\mathcal{A}; \mathcal{P},z)=\sum_{\substack{a \in \mathcal{A} \\ (a, P(z))=1}} 1,
$$
$$
\mathcal{M}_k=\{n: n=p_1 p_2 \cdots p_k,  13<n \leqslant x+6, x^{0.005} \leqslant p_1 < \cdots < p_k, n \equiv 4 (\bmod q) \},
$$
$$
\mathcal{A}^{(k)}=\left\{\frac{p+2}{q} : 7<p \leqslant x, p \equiv -2 (\bmod q), p+6 \in \mathcal{M}_k\right\},
$$
$$
\mathcal{E}=\left\{qmp_1p_2p_3p_4: qmp_1p_2p_3p_4 \leqslant x+2, \left(\frac{x}{q}\right)^{\frac{1}{13}} \leqslant p_{1}<p_{2}<p_3<p_{4}<\left(\frac{x}{q}\right)^{\frac{1}{8.4}}, \left(m, qp_{1}^{-1} P\left(p_{2}\right)\right)=1\right\},
$$
$$
\mathcal{B}=\{n-2 : n \in \mathcal{E}\}, \quad
\mathscr{W}=\left\{\left\{\frac{p+2}{q}, p+6\right\}: 7<p \leqslant x, p \equiv -2 (\bmod q)\right\},
$$
$$
\mathscr{W}^{(1)}=\{\{n-2,n+4\}: n \in \mathcal{E}\}, \quad \mathscr{W}^{(2)}=\left\{\left\{n-6, \frac{n-4}{q}\right\}: n \in \mathcal{M}_k \right\}.
$$

\begin{lemma}\label{l1}
([\cite{Cai2017}, Lemma 1], deduced from \cite{HBL2016}, Proposition 1).
Let $\mathscr{W}$ be a finite subset of $\mathbb{N}^{2}$. Suppose that $z_{1}, z_{2} \geqslant 2$ with $\log z_{1} \asymp \log z_{2}$ and write $\mathbf{z}=\left\{z_{1}, z_{2}\right\}$. For $\mathbf{d}=\left\{d_{1}, d_{2}\right\}$ and $\mathbf{n}=\left\{n_{1}, n_{2}\right\}$, we write $\mathbf{d} \mid \mathbf{n}$ to mean that $d_{1} \mid n_{1}$ and $d_{2} \mid n_{2}$. Set
$$
\mathscr{W}_{\mathbf{d}}=\{\mathbf{n} \in \mathscr{W}: \mathbf{d} \mid \mathbf{n}\}, \quad S(\mathscr{W}, \mathbf{z})=\sum_{\substack{\left\{n_{1}, n_{2}\right\} \in \mathscr{W} \\ p|n_{1} \Rightarrow p \geqslant z_{1} \\ p| n_{2} \Rightarrow p \geqslant z_{2}}} 1 .
$$
Suppose that
$$
\left|\mathscr{W}_{\mathbf{d}}\right|=h(\mathbf{d}) X+R(\mathbf{d})
$$
for some $X>0$ independent of $\mathbf{d}$ and some multiplicative function $h(\mathbf{d}) \in$ $[0,1)$ such that $h(p, 1)+h(1, p)-1<h(p, p) \leqslant h(p, 1)+h(1, p)$ for all primes $p$, and
$$
h(p, 1), h(1, p) \leqslant c p^{-1}, \quad h(p, p) \leqslant c p^{-2}
$$
for some constant $c \geqslant 2$.

Let $h_{1}(d)=h(d, 1)$ and $h_{2}(d)=h(1, d)$. Suppose that
$$
\prod_{w \leqslant p<z}\left(1-h_{j}(p)\right)^{-1} \leqslant \frac{\log z}{\log w}\left(1+\frac{L}{\log w}\right) \quad(j=1,2)
$$
for $z \geqslant w \geqslant 2$ and some positive constant L. Then: 
$$
S(\mathscr{W}, \mathbf{z}) \leqslant X V\left(z_{0}, h^{*}\right) V_{1} V_{2} \left\{\left(F\left(s_{1}\right) F\left(s_{2}\right)\right)\left(1+O\left(\left(\log D_{1} D_{2}\right)^{-1 / 6}\right)\right)\right\}
$$
$$
+O_{\varepsilon}\left(\sum_{d_{1} d_{2} \leqslant\left(D_{1} D_{2}\right)^{1+\varepsilon}} \tau^{4}\left(d_{1} d_{2}\right)\left|R\left(\left\{d_{1}, d_{2}\right\}\right)\right|\right),
$$
$$
S(\mathscr{W}, \mathbf{z}) \geqslant X V\left(z_{0}, h^{*}\right) V_{1} V_{2}\left\{\left(f\left(s_{1}\right) F\left(s_{2}\right)+F\left(s_{1}\right) f\left(s_{2}\right)-F\left(s_{1}\right) F\left(s_{2}\right)\right)\left(1+O\left(\left(\log D_{1} D_{2}\right)^{-1 / 6}\right)\right)\right\}
$$
$$
+O_{\varepsilon}\left(\sum_{d_{1} d_{2} \leqslant\left(D_{1} D_{2}\right)^{1+\varepsilon}} \tau^{4}\left(d_{1} d_{2}\right)\left|R\left(\left\{d_{1}, d_{2}\right\}\right)\right|\right)
$$
for any $\varepsilon>0$, where
$$
z_{0}=\exp \left(\sqrt[3]{\log z_{1} z_{2}}\right), \quad s_{j}=\frac{\log D_{j}}{\log z_{j}} \quad(j=1,2),
$$
$$
V\left(z_{0}, h^{*}\right)=\prod_{p<z_{0}}\left(1-h^{*}(p)\right), \quad h^{*}(p)=h(p, 1)+h(1, p)-h(p, p),
$$
$$
V_{j}=\prod_{z_{0} \leqslant p<z_{j}}\left(1-h_{j}(p)\right) \quad(j=1,2) .
$$
and $\gamma$ denotes the Euler's constant, $f(s)$ and $F(s)$ are determined by the following differential-difference equation
\begin{align*}
\begin{cases}
F(s)=\frac{2 e^{\gamma}}{s}, \quad f(s)=0, \quad &0<s \leqslant 2,\\
(s F(s))^{\prime}=f(s-1), \quad(s f(s))^{\prime}=F(s-1), \quad &s \geqslant 2 .
\end{cases}
\end{align*}
\end{lemma}

\begin{lemma}\label{l2} ([\cite{CAI867}, Lemma 2], deduced from \cite{HR74}).
\begin{align*}
F(s)=&\frac{2 e^{\gamma}}{s}, \quad 0<s \leqslant 3;\\
F(s)=&\frac{2 e^{\gamma}}{s}\left(1+\int_{2}^{s-1} \frac{\log (t-1)}{t} d t\right), \quad 3 \leqslant s \leqslant 5 ;\\
F(s)=&\frac{2 e^{\gamma}}{s}\left(1+\int_{2}^{s-1} \frac{\log (t-1)}{t} d t+\int_{2}^{s-3} \frac{\log (t-1)}{t} d t \int_{t+2}^{s-1} \frac{1}{u} \log \frac{u-1}{t+1} d u\right), \quad 5 \leqslant s \leqslant 7;\\
f(s)=&\frac{2 e^{\gamma} \log (s-1)}{s}, \quad 2 \leqslant s \leqslant 4 ;\\
f(s)=&\frac{2 e^{\gamma}}{s}\left(\log (s-1)+\int_{3}^{s-1} \frac{d t}{t} \int_{2}^{t-1} \frac{\log (u-1)}{u} d u\right), \quad 4 \leqslant s \leqslant 6 ;\\
f(s)=& \frac{2 e^{\gamma}}{s}\left(\log (s-1)+\int_{3}^{s-1} \frac{d t}{t} \int_{2}^{t-1}\frac{\log (u-1)}{u} d u\right.\\ & \left.+\int_{2}^{s-4} \frac{\log (t-1)}{t} d t \int_{t+2}^{s-2} \frac{1}{u} \log \frac{u-1}{t+1} \log \frac{s}{u+2} d u\right), \quad 6 \leqslant s \leqslant 8.
\end{align*}
\end{lemma}

\begin{lemma}\label{l4} ([\cite{CAI867}, Lemma 4], deduced from \cite{JIA96}, \cite{ERPAN}). Let
$$
x>1, \quad z=x^{\frac{1}{u}}, \quad Q(z)=\prod_{p<z} p.
$$
Then for $u \geqslant 1$, we have
$$
\sum_{\substack{n \leqslant x \\(n, Q(z))=1}} 1=w(u) \frac{x}{\log z}+O\left(\frac{x}{\log ^{2} z}\right),
$$
where $w(u)$ is determined by the following differential-difference equation
\begin{align*}
\begin{cases}
w(u)=\frac{1}{u}, & \quad 1 \leqslant u \leqslant 2, \\
(u w(u))^{\prime}=w(u-1), & \quad u \geqslant 2 .
\end{cases}
\end{align*}
Moreover, we have
$$
\begin{cases}w(u) \leqslant \frac{1}{1.763}, & u \geqslant 2, \\ w(u)<0.5644, & u \geqslant 3, \\ w(u)<0.5617, & u \geqslant 4.\end{cases}
$$
\end{lemma}

\begin{conjecture}\label{mean}
(\text{GEH}$(\theta)$, deduced from [\cite{ljm}, Conjecture 1.2]).
Let $\theta \in (0,1)$ be a constant and put
$$
\begin{aligned}
P\left(y_{1}, \ldots, y_{k} ; z_{1}, \ldots, z_{k}\right) & =\left\{m=p_{1} \cdots p_{k} : y_{1} \leqslant p_{1} \leqslant z_{1}, \ldots, y_{k} \leqslant p_{k} \leqslant z_{k}\right\}, \\
\pi_{k}(x ; q, a) & =\sum_{\substack{m \in P\left(y_{1}, \ldots, y_{k} ; z_{1}, \ldots, z_{k}\right) \\
m \leqslant x, m \equiv a(\bmod q)}} 1, \\
\pi_{k}(x ; q) & =\sum_{\substack{m \in P\left(y_{1}, \ldots, y_{k} ; z_{1}, \ldots, z_{k}\right) \\
m \leqslant x,(m, q)=1}} 1 .
\end{aligned}
$$
Then for any $A>0$ and $l \geqslant 1$ there exists $B=B(A, l)>0$ such that
$$
\sum_{q \leqslant x^{\theta} \log ^{-B} x} \tau^{l}(q) \max _{(a, q)=1}\left|\pi_{k}(x ; q, a)-\frac{\pi_{k}(x ; q)}{\varphi(q)}\right| \ll \frac{x}{\log ^{A} x},
$$
where the implied constant depends only on $k, l$ and $A$. 
\end{conjecture}

\begin{lemma}\label{GEHtrue}
([\cite{Cai2017}, Lemma 3], deduced from \cite{HBL2016}, \cite{Opera}).
\text{GEH}$(\theta)$ is true for $\theta \leqslant \frac{1}{2}$. 

\end{lemma}

\section{Weighted sieve method}

\begin{lemma}\label{l3} We have
\begin{align*}
4\pi_{1,3,14}(x) \geqslant &3S\left(\mathcal{A};\mathcal{P}, \left(\frac{x}{q}\right)^{\frac{1}{13}}\right)+S\left(\mathcal{A};\mathcal{P}, \left(\frac{x}{q}\right)^{\frac{1}{8.4}}\right)\\
&+\sum_{(\frac{x}{q})^{\frac{1}{13}} \leqslant p_1<p_2<(\frac{x}{q})^{\frac{1}{8.4}}  }S\left(\mathcal{A}_{p_1 p_2};\mathcal{P},\left(\frac{x}{q}\right)^{\frac{1}{13}}\right)\\
&+\sum_{(\frac{x}{q})^{\frac{1}{13}} \leqslant p_1<(\frac{x}{q})^{\frac{1}{8.4}} \leqslant p_2<(\frac{x}{q})^{0.475-\frac{2}{13}-\varepsilon}p^{-1}_1  }S\left(\mathcal{A}_{p_1 p_2};\mathcal{P},\left(\frac{x}{q}\right)^{\frac{1}{13}}\right)\\
&-\sum_{(\frac{x}{q})^{\frac{1}{13}} \leqslant p<(\frac{x}{q})^{\frac{1}{3.145}} }S\left(\mathcal{A}_{p};\mathcal{P},\left(\frac{x}{q}\right)^{\frac{1}{13}}\right)\\
&-\sum_{(\frac{x}{q})^{\frac{1}{13}} \leqslant p<(\frac{x}{q})^{\frac{1}{3.81}} }S\left(\mathcal{A}_{p};\mathcal{P},\left(\frac{x}{q}\right)^{\frac{1}{13}}\right)\\
&-\sum_{(\frac{x}{q})^{\frac{1}{13}} \leqslant p_1<(\frac{x}{q})^{\frac{1}{3.145}} \leqslant p_2 <(\frac{x}{qp_1})^{\frac{1}{2}} }S\left(\mathcal{A}_{p_1 p_2};\mathcal{P}(p_1),p_2\right)\\
&-\sum_{(\frac{x}{q})^{\frac{1}{8.4}} \leqslant p_1<(\frac{x}{q})^{\frac{1}{3.81}} \leqslant p_2 <(\frac{x}{qp_1})^{\frac{1}{2}} }S\left(\mathcal{A}_{p_1 p_2};\mathcal{P}(p_1),\left(\frac{x}{qp_1 p_2}\right)^{\frac{1}{2}}\right)\\
&-\sum_{(\frac{x}{q})^{\frac{1}{13}} \leqslant p_1 < p_2 < p_3< p_4<(\frac{x}{q})^{\frac{1}{8.4}}  }S\left(\mathcal{A}_{p_1 p_2 p_3 p_4};\mathcal{P}(p_1),p_2\right) \\
&-\sum_{(\frac{x}{q})^{\frac{1}{13}} \leqslant p_1 < p_2 < p_3<(\frac{x}{q})^{\frac{1}{8.4}} \leqslant p_4< (\frac{x}{q})^{0.475-\frac{2}{13}-\varepsilon}p^{-1}_3}S\left(\mathcal{A}_{p_1 p_2 p_3 p_4};\mathcal{P}(p_1),p_2\right) \\
&-2\sum_{(\frac{x}{q})^{\frac{1}{3.145}} \leqslant p_1 < p_2 <(\frac{x}{qp_1})^{\frac{1}{2}} }S\left(\mathcal{A}_{p_1 p_2};\mathcal{P}(p_1),p_2\right)\\
&-2\sum_{(\frac{x}{q})^{\frac{1}{3.81}} \leqslant p_1 < p_2 <(\frac{x}{qp_1})^{\frac{1}{2}} }S\left(\mathcal{A}_{p_1 p_2};\mathcal{P}(p_1),p_2\right)\\
&-\sum_{k=15}^{199} S\left(\mathcal{A}^{(k)} ; \mathcal{P}, \left(\frac{x}{q}\right)^{\frac{1}{13}}\right)-\sum_{k=15}^{199} S\left(\mathcal{A}^{(k)} ; \mathcal{P}, \left(\frac{x}{q}\right)^{\frac{1}{8.4}}\right) \\
& -\sum_{k=15}^{199} S\left(\mathcal{A}^{(k)} ; \mathcal{P}, \left(\frac{x}{q}\right)^{\frac{1}{3.81}}\right)-\sum_{k=15}^{199} S\left(\mathcal{A}^{(k)} ; \mathcal{P}, \left(\frac{x}{q}\right)^{\frac{1}{3.145}}\right)+O\left(x^{\frac{12}{13}}\right) \\
=&\left(3 S_{11}+S_{12}\right)+\left(S_{21}+S_{22}\right)-\left(S_{31}+S_{32}\right)-\left(S_{41}+S_{42}\right)\\
&-\left(S_{51}+S_{52}\right)-2\left(S_{61}+S_{62}\right)-\left(S_{71}+S_{72}+S_{73}+S_{74}\right)+O\left(x^{\frac{12}{13}}\right)\\
=&S_1+S_2-S_3-S_4-S_5-2S_6-S_7+O\left(x^{\frac{12}{13}}\right).
\end{align*}
\end{lemma}

\begin{proof} It is similar to that of [\cite{Cai2017}, Lemma 6] so we omit it here. 
\end{proof}

\section{Proof of Theorem 1.1}
In this section, sets $\mathcal{A}$, $\mathcal{E}$, $\mathcal{B}$, $\mathcal{M}_k$, $\mathcal{A}^{(k)}$, $\mathscr{W}$, $\mathscr{W}^{(1)}$ and $\mathscr{W}^{(2)}$ are defined respectively.

\subsection{Evaluation of $S_{1}, S_{2}, S_{3}$}
We are going to use Lemma~\ref{l1} to the set $\mathscr{W}$ and obtain upper and lower bounds of $S_{1}, S_{2}$ and $S_{3}$. For a prime $p>7$, we note that $d_1 \mid \left(\frac{p+2}{q}\right)$ and $d_2 \mid (p+6)$ imply that $(d_1,d_2)=(d_1,2)=(d_2,6)=1$. Therefore we can take
$$
\begin{aligned}
& \left|\mathscr{W}_{\mathbf{d}}\right|=h(\mathbf{d}) X+R(\mathbf{d}), \quad X=\frac{\pi(x)}{\varphi(q)}, \\
& h(\mathbf{d})= \begin{cases}\frac{1}{\varphi\left(d_{1} d_{2}\right)}, & \left(d_{1}, d_{2}\right)=\left(d_{1}, 2\right)=\left(d_{2}, 6\right)=1, \\
0, & \text {otherwise. }\end{cases}
\end{aligned}
$$

It is easy to show that
$$
h_{1}(p)=\begin{cases}
0, & p=2, \\
\frac{1}{p-1}, & p \geqslant 3 ;
\end{cases} \quad h_{2}(p)=\begin{cases}
0, & p=2,3, \\
\frac{1}{p-1}, & p \geqslant 5 ;
\end{cases} \quad h^{*}(p)= \begin{cases}0, & p=2, \\
\frac{1}{2}, & p=3, \\
\frac{2}{p-1}, & p \geqslant 5.\end{cases}
$$
and
\begin{align}
\nonumber V\left(z_{0}, h^{*}\right) V_{1} V_{2} & =\frac{1}{2} \prod_{3<p \leqslant z_{0}}\left(1-\frac{2}{p-1}\right) \prod_{z_{0}<p \leqslant (\frac{x}{q})^{\frac{1}{13}}}\left(1-\frac{1}{p-1}\right) \prod_{z_{0}<p \leqslant x^{0.005}}\left(1-\frac{1}{p-1}\right) \\
& =\left(1+O\left(z_{0}^{-1}\right)\right) C_3 V\left(\left(\frac{x}{q}\right)^{\frac{1}{13}}\right) V(x^{0.005})
\end{align}
with
$$
V(z)=\prod_{p<z}\left(1-\frac{1}{p}\right)=\frac{e^{-\gamma}}{\log z}\left(1+O\left(\frac{1}{\log z}\right)\right).
$$

To deal with the error term, by the Chinese remainder theorem, we have
$$
|R(\mathbf{d})| \leqslant\left|r(d_{1} d_{2})\right|,
$$
where
$$
|r(d)|=\max _{(a, d)=1}\left|\sum_{\substack{p \leqslant x \\ p \equiv a(\bmod d)}} 1-\frac{\pi(x)}{\varphi(d)}\right|+O(1) .
$$
By Bombieri's theorem we have
\begin{equation}
\sum_{d_{1} d_{2} \leqslant x^{\frac{1}{2}-\varepsilon}} \tau^{4}\left(d_{1} d_{2}\right)|R(\mathbf{d})| \ll x(\log x)^{-5} .
\end{equation}

Then by Lemma~\ref{l1} we have
\begin{align}
\nonumber S_{11}= & S\left(\mathscr{W},\left\{\left(\frac{x}{q}\right)^{\frac{1}{13}}, x^{0.005}\right\}\right) \\
\nonumber \geqslant & (1+o(1)) \frac{\pi(x)}{\varphi(q)} C_3 V\left(\left(\frac{x}{q}\right)^{\frac{1}{13}}\right) V\left(x^{0.005}\right)\{f(6.175) F(5)+F(6.175) f(5)-F(6.175) F(5)\} \\
\nonumber & +O\left(x(\log x)^{-5}\right) \\
\nonumber = & (1+o(1)) \frac{4 C_3 \pi(x)}{\varphi(q)\left(\log x^{0.475-\varepsilon}\right)\left(\log x^{0.025}\right)} \left\{f_{0}(6.175) F_{0}(5)+F_{0}(6.175) f_{0}(5)-F_{0}(6.175) F_{0}(5)\right\} \\
\nonumber & +O\left(x(\log x)^{-5}\right) \\
\geqslant & 818.10189 \frac{C_3 x \log \log x}{ (\log x)^3},
\end{align}
where
$$
f_{0}(s)=\frac{s}{2 e^{\gamma}} f(s), \quad F_{0}(s)=\frac{s}{2 e^{\gamma}} F(s) .
$$

Similarly, we have
\begin{align}
\nonumber S_{12}= & S\left(\mathscr{W},\left\{\left(\frac{x}{q}\right)^{\frac{1}{8.4}}, x^{0.005}\right\}\right) \\
\nonumber \geqslant & (1+o(1)) \frac{\pi(x)}{\varphi(q)} C_3 V\left(\left(\frac{x}{q}\right)^{\frac{1}{8.4}}\right) V\left(x^{0.005}\right)\{f(3.99) F(5)+F(3.99) f(5)-F(3.99) F(5)\} \\
\nonumber & +O\left(x(\log x)^{-5}\right) \\
\nonumber = & (1+o(1)) \frac{4 C_3 \pi(x)}{\varphi(q)\left(\log x^{0.475-\varepsilon}\right)\left(\log x^{0.025}\right)} \left\{f_{0}(3.99) F_{0}(5)+F_{0}(3.99) f_{0}(5)-F_{0}(3.99) F_{0}(5)\right\} \\
\nonumber & +O\left(x(\log x)^{-5}\right) \\
\geqslant & 516.86063 \frac{C_3 x \log \log x}{ (\log x)^3},\\
\nonumber S_{21} \geqslant & (1+o(1)) \frac{0.475 \times 4 C_3 \pi(x)}{\varphi(q)\left(\log x^{0.475-\varepsilon}\right)\left(\log x^{0.025}\right)} \left\{f_{0}(5) G+F_{0}(5) g-F_{0}(5) G\right\}+O\left(x(\log x)^{-5}\right) \\
\geqslant & 73.9301 \frac{C_3 x \log \log x}{ (\log x)^3},\\
\nonumber S_{22} \geqslant & (1+o(1)) \frac{0.475 \times 4 C_3 \pi(x)}{\varphi(q)\left(\log x^{0.475-\varepsilon}\right)\left(\log x^{0.025}\right)} \left\{f_{0}(5) H+F_{0}(5) h-F_{0}(5) H\right\}+O\left(x(\log x)^{-5}\right) \\
\geqslant & 149.13684 \frac{C_3 x \log \log x}{ (\log x)^3},\\
\nonumber S_{31} \leqslant & (1+o(1)) \frac{0.475 \times 4 C_3 \pi(x)}{\varphi(q)\left(\log x^{0.475-\varepsilon}\right)\left(\log x^{0.025}\right)} \left\{F_{0}(5) J\right\}+O\left(x(\log x)^{-5}\right) \\
\leqslant & 1282.38485 \frac{C_3 x \log \log x}{ (\log x)^3},\\
\nonumber S_{32} \leqslant & (1+o(1)) \frac{0.475 \times 4 C_3 \pi(x)}{\varphi(q)\left(\log x^{0.475-\varepsilon}\right)\left(\log x^{0.025}\right)} \left\{F_{0}(5) K\right\}+O\left(x(\log x)^{-5}\right) \\
\leqslant & 1048.20211 \frac{C_3 x \log \log x}{ (\log x)^3},
\end{align}
where
\begin{align}
\nonumber G= & \int_{1 / 13}^{1 / 8.4} \frac{d t_{1}}{t_{1}} \int_{t_{1}}^{1 / 8.4} \frac{d t_{2}}{t_{2}\left(0.475-t_{1}-t_{2}\right)} \\
\nonumber & +\int_{1 / 13}^{1 / 8.4} \frac{d t_{1}}{t_{1}} \int_{t_{1}}^{1 / 8.4} \frac{d t_{2}}{t_{2}\left(0.475-t_{1}-t_{2}\right)} \int_{2}^{5.175-13\left(t_{1}+t_{2}\right)} \frac{\log \left(t_{3}-1\right)}{t_{3}} d t_{3}, \\
\nonumber g= & \int_{1 / 13}^{1 / 8.4} \frac{d t_{1}}{t_{1}} \int_{t_{1}}^{1 / 8.4} \frac{\log \left(5.175-13\left(t_{1}+t_{2}\right)\right)}{t_{2}\left(0.475-t_{1}-t_{2}\right)} d t_{2} \\
\nonumber & +\int_{1 / 13}^{2.175 / 26} \frac{d t_{1}}{t_{1}} \int_{t_{1}}^{2.175 / 13-t_{1}} \frac{d t_{2}}{t_{2}\left(0.475-t_{1}-t_{2}\right)} \int_{3}^{5.175-13\left(t_{1}+t_{2}\right)} \frac{d t_{3}}{t_{3}} \int_{2}^{t_{3}-1} \frac{\log \left(t_{4}-1\right)}{t_{4}} d t_{4} , \\
\nonumber H= & \int_{1 / 13}^{1 / 8.4} \frac{d t_{1}}{t_{1}} \int_{1 / 8.4}^{0.475-2 / 13-t_{1}} \frac{d t_{2}}{t_{2}\left(0.475-t_{1}-t_{2}\right)} \\
\nonumber & +\int_{1 / 13}^{1 / 8.4} \frac{d t_{1}}{t_{1}} \int_{1 / 8.4}^{0.475-2 / 13-t_{1}} \frac{d t_{2}}{t_{2}\left(0.475-t_{1}-t_{2}\right)} \int_{2}^{5.175-13\left(t_{1}+t_{2}\right)} \frac{\log \left(t_{3}-1\right)}{t_{3}} d t_{3}, \\
\nonumber h= & \int_{1 / 13}^{1 / 8.4} \frac{d t_{1}}{t_{1}} \int_{1 / 8.4}^{0.475-2 / 13-t_{1}} \frac{\log \left(5.175-13\left(t_{1}+t_{2}\right)\right)}{t_{2}\left(0.475-t_{1}-t_{2}\right)} d t_{2} ,\\
\nonumber J= & \int_{1 / 13}^{1 / 3.145} \frac{d t}{t(0.475-t)}+\int_{1 / 13}^{0.475-3 / 13} \frac{d t_{1}}{t_{1}\left(0.475-t_{1}\right)} \int_{2}^{5.175-13 t_{1}} \frac{\log \left(t_{2}-1\right)}{t_{2}} d t_{2} \\
\nonumber & +\int_{1 / 13}^{0.475-5 / 13} \frac{d t_{1}}{t_{1}\left(0.475-t_{1}\right)} \int_{2}^{3.175-13 t} \frac{\log \left(t_{2}-1\right)}{t_{2}} d t_{2} \int_{t_{2}+2}^{5.175} \frac{1}{t_{3}} \log \frac{t_{3}-1}{t_{2}+1} d t_{3} , \\
\nonumber K= & \int_{1 / 13}^{1 / 3.81} \frac{d t}{t(0.475-t)}+\int_{1 / 13}^{0.475-3 / 13} \frac{d t_{1}}{t_{1}\left(0.475-t_{1}\right)} \int_{2}^{5.175-13 t_{1}} \frac{\log \left(t_{2}-1\right)}{t_{2}} d t_{2} \\
\nonumber & +\int_{1 / 13}^{0.475-5 / 13} \frac{d t_{1}}{t_{1}\left(0.475-t_{1}\right)} \int_{2}^{3.175-13 t} \frac{\log \left(t_{2}-1\right)}{t_{2}} d t_{2} \int_{t_{2}+2}^{5.175} \frac{1}{t_{3}} \log \frac{t_{3}-1}{t_{2}+1} d t_{3} .
\end{align}

\subsection{Evaluation of $S_{4}, S_{5}, S_{6}$}
By Chen's role-reversal trick we know that
\begin{align}
\nonumber S_{51}=&\sum_{\substack{p \in \mathcal{B} \\ \left(p+6, P\left(x^{0.005}\right)\right)=1}} 1 \\
\nonumber \leqslant& \sum_{\substack{m \in \mathcal{B} \\\left(m, P\left(x^{\frac{0.475-\varepsilon}{2}}\right)\right)=1 \\ \left(m+6, P\left(x^{0.005}\right)\right)=1}} 1+O\left(x^{\frac{1}{2}}\right) \\
=& S\left(\mathscr{W}^{(1)},\left\{x^{\frac{0.475-\varepsilon}{2}}, x^{0.005}\right\}\right)+O\left(x^{\frac{1}{2}}\right).
\end{align}
we may write
$$
\left|\mathscr{W}_{\mathbf{d}}^{(1)}\right|=h(\mathbf{d})|\mathcal{E}|+R^{(1)}(\mathbf{d}),
$$
where
$$
h(\mathbf{d})= \begin{cases}\frac{1}{\varphi\left(d_{1} d_{2}\right)}, & \left(d_{1}, d_{2}\right)=\left(d_{1}, 2\right)=\left(d_{2}, 6\right)=1, \\
0, & \text {otherwise }\end{cases}
$$
and
$$
\begin{aligned}
\left|R^{(1)}(\mathbf{d})\right| \leqslant & \max _{\left(a, d_{1} d_{2}\right)=1}\left|\sum_{\substack{n \in \mathcal{E} \\
n \equiv a\left(\bmod d_{1} d_{2}\right)}} 1-\frac{1}{\varphi\left(d_{1} d_{2}\right)} \sum_{\substack{n \in \mathcal{E} \\
\left(n, d_{1} d_{2}\right)=1}} 1\right|+\frac{1}{\varphi\left(d_{1} d_{2}\right)} \sum_{\substack{n \in \mathcal{E} \\
\left(n, d_{1} d_{2}\right)>1}} 1 \\
=& R_{1}^{(1)}(\mathbf{d})+R_{2}^{(1)}(\mathbf{d}) .
\end{aligned}
$$

To deal with the error term, by the arguments used in \cite{ljm}, we have
\begin{equation}
\sum_{d_{1} d_{2} \leqslant x^{\frac{1}{2}-\varepsilon}} \tau^{4}\left(d_{1} d_{2}\right)R^{(1)}(\mathbf{d}) \ll x(\log x)^{-5} .
\end{equation}

By Lemma~\ref{l4} we have
\begin{align}
\nonumber |\mathcal{E}|&=\sum_{(\frac{x}{q})^{\frac{1}{13}} \leqslant p_1 < p_2 < p_3< p_4<(\frac{x}{q})^{\frac{1}{8.4}}  }\sum_{\substack{1\leqslant m\leqslant \frac{x}{qp_1 p_2 p_3 p_4}\\\left(m, qp_{1}^{-1}  P\left(p_{2}\right)\right)=1 }}1\\
\nonumber &\leqslant (1+o(1))\frac{0.5617x}{q\log x}\int_{\frac{1}{13}}^{\frac{1}{8.4}} \frac{d t_{1}}{t_{1}} \int_{t_{1}}^{\frac{1}{8.4}} \frac{1}{t_{2}}\left(\frac{1}{t_{1}}-\frac{1}{t_{2}}\right) \log \frac{1}{8.4 t_{2}} d t_{2} .\\
& \leqslant \frac{0.00934x}{q\log x}.
\end{align}

Then by Lemma~\ref{l1} we have
\begin{align}
\nonumber S_{51} \leqslant & S\left(\mathscr{W}^{(1)},\left\{x^{\frac{0.475-\varepsilon}{2}}, x^{0.005}\right\}\right) \\ 
\nonumber \leqslant & (1+o(1)) \frac{ 4 C_3 |\mathcal{E}|}{\left(\log x^{0.475-\varepsilon}\right)\left(\log x^{0.025}\right)} \left\{F_{0}(2) F_{0}(5)\right\}+O\left(x(\log x)^{-5}\right) \\
\leqslant & 4.41937 \frac{C_3 x \log \log x}{ (\log x)^3}.
\end{align}

Similarly, we have
\begin{align}
\nonumber S_{52} \leqslant & (1+o(1)) \frac{ 4 C_3 x L}{q \log x \left(\log x^{0.475-\varepsilon}\right)\left(\log x^{0.025}\right)} \left\{F_{0}(2) F_{0}(5)\right\}+O\left(x(\log x)^{-5}\right) \\
\leqslant & 22.91504 \frac{C_3 x \log \log x}{ (\log x)^3},\\
\nonumber S_{41} \leqslant & (1+o(1)) \frac{ 4 C_3 x\left(
\int_{2.145}^{12} \frac{\log \left(2.145-\frac{3.145}{t+1}\right)}{t} d t
\right)}{q \log x \left(\log x^{0.475-\varepsilon}\right)\left(\log x^{0.025}\right)} \left\{F_{0}(2) F_{0}(5)\right\}+O\left(x(\log x)^{-5}\right) \\
\leqslant & 371.11243 \frac{C_3 x \log \log x}{ (\log x)^3},\\
\nonumber S_{42} \leqslant & (1+o(1)) \frac{ 4 C_3 x \left(
\int_{2.81}^{7.4} \frac{\log \left(2.81-\frac{3.81}{t+1}\right)}{t} d t
\right)}{q \log x \left(\log x^{0.475-\varepsilon}\right)\left(\log x^{0.025}\right)} \left\{F_{0}(2) F_{0}(5)\right\}+O\left(x(\log x)^{-5}\right) \\
\leqslant & 341.31874 \frac{C_3 x \log \log x}{ (\log x)^3},\\
\nonumber S_{61} \leqslant & (1+o(1)) \frac{ 4 C_3 x \left(
\int_{2}^{2.145} \frac{\log \left(t-1\right)}{t} d t
\right)}{q \log x \left(\log x^{0.475-\varepsilon}\right)\left(\log x^{0.025}\right)} \left\{F_{0}(2) F_{0}(5)\right\}+O\left(x(\log x)^{-5}\right) \\
\leqslant & 2.27032 \frac{C_3 x \log \log x}{ (\log x)^3},\\
\nonumber S_{62} \leqslant & (1+o(1)) \frac{ 4 C_3 x \left(
\int_{2}^{2.81} \frac{\log \left(t-1\right)}{t} d t
\right)}{q \log x \left(\log x^{0.475-\varepsilon}\right)\left(\log x^{0.025}\right)} \left\{F_{0}(2) F_{0}(5)\right\}+O\left(x(\log x)^{-5}\right) \\
\leqslant & 49.78864 \frac{C_3 x \log \log x}{ (\log x)^3},
\end{align}
where
\begin{align}
\nonumber L &=\int_{1 / 13}^{1 / 8.4} \frac{d t_{1}}{t_{1}} \int_{t_{1}}^{1 / 8.4} \frac{d t_{2}}{t_{2}^{2}} \int_{t_{2}}^{1 / 8.4} \frac{d t_{3}}{t_{3}} \int_{1 / 8.4}^{0.475-2 / 13-t_{3}} \frac{w\left(\frac{1-t_{1}-t_{2}-t_{3}-t_{4}}{t_{2}}\right)}{t_{4}} d t_{4} \\
\nonumber & \leqslant 0.5644 \int_{1 / 13}^{1 / 8.4} \frac{d t_{1}}{t_{1}} \int_{t_{1}}^{1 / 8.4} \frac{1}{t_{2}}\left(\frac{1}{t_{1}}-\frac{1}{t_{2}}\right) \log 8.4\left(0.475-\frac{2}{13}-t_{3}\right) d t_{2} \\
\nonumber & \leqslant 0.04839.
\end{align}

\subsection{Evaluation of $S_{7}$}
By Chen's role-reversal trick we know that
\begin{align}
\nonumber S_{71}=&\sum_{\substack{p+6 \in \mathcal{M}_k \\ \left(\frac{p+2}{q}, P\left(\left(\frac{x}{q}\right)^{\frac{1}{13}}\right)\right)=1}} 1 \\
\nonumber \leqslant& \sum_{\substack{m \in \mathcal{M}_k \\ \left(m-6, P\left(x^{\frac{0.025}{2}}\right)\right)=1 \\ \left(\frac{m-4}{q}, P\left(\left(\frac{x}{q}\right)^{\frac{1}{13}}\right)\right)=1}} 1+O\left(x^{\frac{1}{2}}\right) \\
=& S\left(\mathscr{W}^{(2)},\left\{x^{\frac{0.025}{2}}, \left(\frac{x}{q}\right)^{\frac{1}{13}}\right\}\right)+O\left(x^{\frac{1}{2}}\right).
\end{align}
we may write
$$
\left|\mathscr{W}_{\mathbf{d}}^{(2)}\right|=h(\mathbf{d})|\mathcal{M}_k|+R^{(2)}(\mathbf{d}),
$$
where
$$
h(\mathbf{d})= \begin{cases}\frac{1}{\varphi\left(d_{1} d_{2}\right)}, & \left(d_{1}, d_{2}\right)=\left(d_{1}, 2\right)=\left(d_{2}, 6\right)=1, \\
0, & \text {otherwise }\end{cases}
$$
and
$$
\begin{aligned}
\left|R^{(2)}(\mathbf{d})\right| \leqslant & \max _{\left(a, d_{1} d_{2}\right)=1}\left|\sum_{\substack{n \in \mathcal{M}_k \\
n \equiv a\left(\bmod d_{1} d_{2}\right)}} 1-\frac{1}{\varphi\left(d_{1} d_{2}\right)} \sum_{\substack{n \in \mathcal{M}_k \\
\left(n, d_{1} d_{2}\right)=1}} 1\right|+\frac{1}{\varphi\left(d_{1} d_{2}\right)} \sum_{\substack{n \in \mathcal{M}_k \\
\left(n, d_{1} d_{2}\right)>1}} 1 \\
=& R_{1}^{(2)}(\mathbf{d})+R_{2}^{(2)}(\mathbf{d}) .
\end{aligned}
$$

To deal with the error term, by the arguments similar to those for $S_{51}$, we have
$$
\sum_{d_{1} d_{2} \leqslant x^{\frac{1}{2}-\varepsilon}} \tau^{4}\left(d_{1} d_{2}\right)R^{(2)}(\mathbf{d}) \ll x(\log x)^{-5} .
$$

By the prime number theorem and summation by parts we have
\begin{align}
\nonumber \left|\mathcal{M}_{k}\right| &=\frac{1}{\varphi(q)} \sum_{z \leqslant p_{1} \leqslant \cdots \leqslant p_{k-1} \leqslant\left(\frac{x+6}{p_{1} \cdots p_{k-2}}\right)^{1 / 2}} \frac{x}{p_{1} \cdots p_{k-1} \log \frac{x}{p_{1} \cdots p_{k-1}}} \\
& =\left(1+O\left(\frac{1}{\log x}\right)\right) c_{k} \frac{\pi(x)}{\varphi(q)},
\end{align}
where
$$
c_{k}=\int_{k-1}^{199} \frac{d t_{1}}{t_{1}} \int_{k-2}^{t_{1}-1} \frac{d t_{2}}{t_{2}} \cdots \int_{3}^{t_{k-4}-1} \frac{d t_{k-3}}{t_{k-3}} \int_{2}^{t_{k-3}-1} \frac{\log \left(t_{k-2}-1\right) d t_{k-2}}{t_{k-2}} .
$$
By similar numerical integration used in \cite{Cai2017}, we have
\begin{equation}
C_0=\sum_{k=15}^{199} c_{k}<0.00408.
\end{equation}

Then from (29)--(31) we have
\begin{align}
\nonumber S_{71} \leqslant & (1+o(1)) \frac{ 4 C_0 C_3 \pi(x)}{\varphi(q) \left(\log x^{0.475-\varepsilon}\right)\left(\log x^{0.025}\right)} \left\{F_{0}(2) F_{0}(6.175)\right\}+O\left(x(\log x)^{-5}\right) \\
\leqslant & 2.38485 \frac{C_3 x \log \log x}{ (\log x)^3}.
\end{align}

Similarly, we have
\begin{align}
\nonumber S_{72} \leqslant & (1+o(1)) \frac{ 4 C_0 C_3 \pi(x)}{\varphi(q) \left(\log x^{0.475-\varepsilon}\right)\left(\log x^{0.025}\right)} \left\{F_{0}(2) F_{0}(3.99)\right\}+O\left(x(\log x)^{-5}\right) \\
\leqslant & 1.57643 \frac{C_3 x \log \log x}{ (\log x)^3},\\
\nonumber S_{73} \leqslant & (1+o(1)) \frac{ 4 C_0 C_3 \pi(x)}{\varphi(q) \left(\log x^{0.475-\varepsilon}\right)\left(\log x^{0.025}\right)} \left\{F_{0}(2) F_{0}(2)\right\}+O\left(x(\log x)^{-5}\right) \\
\leqslant & 1.37432 \frac{C_3 x \log \log x}{ (\log x)^3},\\
\nonumber S_{74} \leqslant & (1+o(1)) \frac{ 4 C_0 C_3 \pi(x)}{\varphi(q) \left(\log x^{0.475-\varepsilon}\right)\left(\log x^{0.025}\right)} \left\{F_{0}(2) F_{0}(2)\right\}+O\left(x(\log x)^{-5}\right) \\
\leqslant & 1.37432 \frac{C_3 x \log \log x}{ (\log x)^3}.
\end{align}

\subsection{Proof of theorem 1.1}
By (14)--(19), (23)--(28) and (32)--(35) we get
$$
S_{1}+S_{2} \geqslant 3194.23324 \frac{C_3 x \log \log x}{(\log x)^{3}},
$$
$$
S_{3}+S_{4}+S_{5}+2S_{6}+S_{7} \leqslant 3181.18071 \frac{C_3 x \log \log x}{(\log x)^{3}},
$$
$$
4\pi_{1,3,14}(x) \geqslant (S_{1}+S_{2})-(S_{3}+S_{4}+S_{5}+2S_{6}+S_{7}) \geqslant 13.05253 \frac{C_3 x \log \log x}{(\log x)^{3}},
$$
$$
\pi_{1,3,14}(x) \geqslant 3.26313 \frac{C_3 x \log \log x}{(\log x)^{3}}.
$$
Now the proof of $\pi_{1,3,14}(x) \gg \frac{C_3 x \log \log x}{(\log x)^3}$ is completed. Then we can prove Theorem~\ref{t1} by replacing $q$ by products of small primes $q_1 q_2 \cdots q_{a-1}$ where $q_i$ denote a prime number satisfies
$$
a \geqslant 2, \quad  q_i < x^{\varepsilon}\ \operatorname{for\ every}\ 1 \leqslant i \leqslant a-1.
$$

\section{An upper bound result}
Now we finish this paper with a look at the upper bound estimate. Let
$$
\mathscr{W}^{\prime}=\{\{p+2, p+6\}: 7<p \leqslant x\},
$$
then we have
\begin{equation}
\pi_{1,1,1}(x) \leqslant S\left(\mathscr{W}^{\prime},\left\{x^{\frac{1}{10}}, x^{\frac{1}{10}}\right\}\right)+O\left(x^{\frac{1}{10}}\right).
\end{equation}

By Lemma~\ref{l1} and some routine arguments we have
\begin{align}
\nonumber & S\left(\mathscr{W}^{\prime},\left\{x^{\frac{1}{10}}, x^{\frac{1}{10}}\right\}\right) \\
\nonumber \leqslant & (1+o(1))  C_3 \pi(x) V\left(x^{\frac{1}{10}}\right)V\left(x^{\frac{1}{10}}\right)\{F(2)F(2)\} +O\left(x(\log x)^{-5}\right) \\
\leqslant & 100 \frac{C_3 x}{ (\log x)^3}.
\end{align}

Finally by (36)--(37) we get the following theorem of the upper bound orders of Hardy-Littlewood prime triples.
\begin{theorem}
$$
\pi_{1,1,1}(x) \ll \frac{C_3 x}{ (\log x)^3} \quad and \quad 
D_{1,1,1}(N) \ll \frac{N}{(\log N)^3}.
$$
\end{theorem}

\bibliographystyle{plain}
\bibliography{bib}
\end{document}